\documentclass{amsart}
\usepackage{amsmath}
\usepackage{amsfonts}

\newtheorem{theorem}{Theorem}
\theoremstyle{plain}

\newtheorem{corollary}{Corollary}

\newtheorem{lemma}{Lemma}

\newtheorem{remark}{Remark}

\numberwithin{equation}{section}
\input{tcilatex}

\begin{document}
\title[Ostrowski-type inequality]{On the weighted Ostrowski type integral
inequality for double integrals}
\author{Mehmet Zeki SARIKAYA}
\address{Department of Mathematics, \ Faculty of Science and Arts, D\"{u}zce
University, D\"{u}zce-TURKEY}
\email{sarikayamz@gmail.com}
\author{Hasan Ogunmez}
\address{Department of Mathematics, \ Faculty of Science and Arts, Afyon
Kocatepe University, Afyon-TURKEY}
\email{hogunmez@aku.edu.tr}

\begin{abstract}
In this paper, we establish new an inequality of weighted Ostrowski-type for
double integrals involving functions of two independent variables by using
fairly elementary analysis.
\end{abstract}

\maketitle

\section{Introduction}

In 1938, the classical integral inequality established by Ostrowski \cite%
{Ostrowski} as follows:

\begin{theorem}
\label{z1} Let $f:[a,b]\mathbb{\rightarrow R}$ be a differentiable mapping
on $(a,b)$ whose derivative $f^{^{\prime }}:(a,b)\mathbb{\rightarrow R}$ is
bounded on $(a,b),$ i.e., $\left\Vert f^{\prime }\right\Vert _{\infty }=%
\underset{t\in (a,b)}{\overset{}{\sup }}\left\vert f^{\prime }(t)\right\vert
<\infty .$ Then, the inequality holds:%
\begin{equation}
\left\vert f(x)-\frac{1}{b-a}\int\limits_{a}^{b}f(t)dt\right\vert \leq \left[
\frac{1}{4}+\frac{(x-\frac{a+b}{2})^{2}}{(b-a)^{2}}\right] (b-a)\left\Vert
f^{\prime }\right\Vert _{\infty }  \label{1}
\end{equation}%
for all$\ x\in \lbrack a,b].$ The constant $\frac{1}{4}$ is the best
possible.
\end{theorem}

In a recent paper \cite{Barnett}, Barnett and Dragomir proved the following
Ostrowski type inequality for double integrals

\begin{theorem}
Let $f:[a,b]\times \lbrack c,d]\mathbb{\rightarrow R}$ be continuous on $%
[a,b]\times \lbrack c,d],$ $f_{x,y}^{\prime \prime }=\frac{\partial ^{2}f}{%
\partial x\partial y}$ exists on $(a,b)\times (c,d)$ and is bounded$,$ i.e., 
\begin{equation*}
\left\Vert f_{x,y}^{\prime \prime }\right\Vert _{\infty }=\underset{(x,y)\in
(a,b)\times (c,d)}{\overset{}{\sup }}\left\vert \frac{\partial ^{2}f(x,y)}{%
\partial x\partial y}\right\vert <\infty .
\end{equation*}
Then, we have the inequality:%
\begin{eqnarray}
&&\left\vert
\int\limits_{a}^{b}\int\limits_{c}^{d}f(s,t)dtds-(d-c)(b-a)f(x,y)\right. 
\notag \\
&&  \notag \\
&&\ \ \ \ \ \ \ \ \ \ \ \ \ \ \ \ \ \ \ \ \ \ \ \ \ \ \ \left. -\left[
(b-a)\int\limits_{c}^{d}f(x,t)dt+(d-c)\int\limits_{a}^{b}f(s,y)ds\right]
\right\vert  \label{2} \\
&&  \notag \\
&\leq &\left[ \frac{1}{4}(b-a)^{2}+(x-\frac{a+b}{2})^{2}\right] \left[ \frac{%
1}{4}(d-c)^{2}+(y-\frac{d+c}{2})^{2}\right] \left\Vert f_{x,y}^{\prime
\prime }\right\Vert _{\infty }  \notag
\end{eqnarray}%
for all $(x,y)\in \lbrack a,b]\times \lbrack c,d].$
\end{theorem}

In \cite{Barnett}, the inequality (\ref{2}) is established by the use of
integral identity involving Peano kernels. In \cite{Pachpatte1}, Pachpatte
obtained an inequality in the view (\ref{2}) by using elementary analysis.
The interested reader is also refered to (\cite{Barnett}, \cite{Dragomir}, 
\cite{Pachpatte}-\cite{Ujevic}) for Ostrowski type inequalities \ in several
independent variables.

The main aim of this note is to establish a new weighted Ostrowski type
inequality for double integrals involving functions of two independent
variables and their partial derivatives.

\section{Main Result}

Throughout this work, we assume that the weight function $w:\left[ a,b\right]
\rightarrow \lbrack 0,\infty ),$ is integrable, nonnegative and 
\begin{equation}
m(a,b)=\int\limits_{a}^{b}w(t)dt<\infty .  \label{3}
\end{equation}

\begin{lemma}
\label{z} Let $f:[a,b]\times \lbrack c,d]\mathbb{\rightarrow R}$ be an
absolutely continuous fuction such that the partial derivative of order $2$
exist for all $(t,s)\in \lbrack a,b]\times \lbrack c,d].$ Then, we have%
\begin{eqnarray}
f(x,y) &=&\frac{1}{m(a,b)}\dint\limits_{a}^{b}w(t)f(t,y)dt+\frac{1}{m(c,d)}%
\dint\limits_{c}^{d}w(s)f(x,s)ds  \notag \\
&&  \notag \\
&&-\frac{1}{m(a,b)m(c,d)}\left[ \dint\limits_{a}^{b}\dint%
\limits_{c}^{d}w(t)w(s)f(t,s)dsdt-\dint\limits_{a}^{b}\dint%
\limits_{c}^{d}p(x,t)q(y,s)\frac{\partial ^{2}f(t,s)}{\partial t\partial s}%
dsdt\right]   \label{4}
\end{eqnarray}%
where%
\begin{equation*}
p(x,t)=\left\{ 
\begin{array}{ll}
p_{1}(a,t)=\dint\limits_{a}^{t}w(u)du, & a\leq t<x \\ 
&  \\ 
p_{2}(b,t)=\dint\limits_{b}^{t}w(u)du, & x\leq t\leq b%
\end{array}%
\right. 
\end{equation*}%
and%
\begin{equation*}
q(y,s)=\left\{ 
\begin{array}{ll}
q_{1}(c,s)=\dint\limits_{c}^{s}w(u)du, & c\leq s<y \\ 
&  \\ 
q_{2}(d,s)=\dint\limits_{d}^{s}w(u)du, & y\leq s\leq d.%
\end{array}%
\right. 
\end{equation*}
\end{lemma}

\begin{proof}
By definitions of $p(x,t)$ and $q(y,s),$ we have%
\begin{equation}
\begin{array}{l}
\dint\limits_{a}^{b}\dint\limits_{c}^{d}p(x,t)q(y,s)\dfrac{\partial
^{2}f(t,s)}{\partial t\partial s}dsdt=\dint\limits_{a}^{x}\dint%
\limits_{c}^{y}p_{1}(a,t)q_{1}(c,s)\dfrac{\partial ^{2}f(t,s)}{\partial
t\partial s}dsdt \\ 
\\ 
+\dint\limits_{a}^{x}\dint\limits_{y}^{d}p_{1}(a,t)q_{2}(d,s)\dfrac{\partial
^{2}f(t,s)}{\partial t\partial s}dsdt+\dint\limits_{x}^{b}\dint%
\limits_{c}^{y}p_{2}(b,t)q_{1}(c,s)\dfrac{\partial ^{2}f(t,s)}{\partial
t\partial s}dsdt \\ 
\\ 
+\dint\limits_{x}^{b}\dint\limits_{y}^{d}p_{2}(b,t)q_{2}(d,s)\dfrac{\partial
^{2}f(t,s)}{\partial t\partial s}dsdt.%
\end{array}
\label{5}
\end{equation}%
Integrating by parts, we can state:%
\begin{equation}
\begin{array}{l}
\dint\limits_{a}^{x}\dint\limits_{c}^{y}p_{1}(a,t)q_{1}(c,s)\dfrac{\partial
^{2}f(t,s)}{\partial t\partial s}dsdt \\ 
\\ 
=\dint\limits_{a}^{x}p_{1}(a,t)\left[ q_{1}(c,y)\dfrac{\partial f(t,y)}{%
\partial t}-\dint\limits_{c}^{y}w(s)\dfrac{\partial f(t,s)}{\partial t}ds%
\right] dt \\ 
\\ 
=p_{1}(a,x)q_{1}(c,y)f(x,y)-q_{1}(c,y)\dint\limits_{a}^{x}w(t)f(t,y)dt \\ 
\\ 
-p_{1}(a,x)\dint\limits_{c}^{y}w(s)f(x,s)ds+\dint\limits_{a}^{x}\dint%
\limits_{c}^{y}w(s)w(t)f(t,s)dsdt.%
\end{array}
\label{6}
\end{equation}%
\begin{equation}
\begin{array}{l}
\dint\limits_{a}^{x}\dint\limits_{y}^{d}p_{1}(a,t)q_{2}(d,s)\dfrac{\partial
^{2}f(t,s)}{\partial t\partial s}dsdt \\ 
\\ 
=\dint\limits_{a}^{x}p_{1}(a,t)\left[ -q_{2}(d,y)\dfrac{\partial f(t,y)}{%
\partial t}-\dint\limits_{y}^{d}w(s)\dfrac{\partial f(t,s)}{\partial t}ds%
\right] dt \\ 
\\ 
=-p_{1}(a,x)q_{2}(d,y)f(x,y)+q_{2}(d,y)\dint\limits_{a}^{x}w(t)f(t,y)dt \\ 
\\ 
-p_{1}(a,x)\dint\limits_{y}^{d}w(s)f(x,s)ds+\dint\limits_{a}^{x}\dint%
\limits_{y}^{d}w(s)w(t)f(t,s)dsdt.%
\end{array}
\label{7}
\end{equation}%
\begin{equation}
\begin{array}{l}
\dint\limits_{x}^{b}\dint\limits_{c}^{y}p_{2}(b,t)q_{1}(c,s)\dfrac{\partial
^{2}f(t,s)}{\partial t\partial s}dsdt \\ 
\\ 
=\dint\limits_{x}^{b}p_{2}(b,t)\left[ q_{1}(c,y)\dfrac{\partial f(t,y)}{%
\partial t}-\dint\limits_{c}^{y}w(s)\dfrac{\partial f(t,s)}{\partial t}ds%
\right] dt \\ 
\\ 
=-p_{2}(b,x)q_{1}(c,y)f(x,y)-q_{1}(c,y)\dint\limits_{x}^{b}w(t)f(t,y)dt \\ 
\\ 
+p_{2}(b,x)\dint\limits_{c}^{y}w(s)f(x,s)ds+\dint\limits_{x}^{b}\dint%
\limits_{c}^{y}w(s)w(t)f(t,s)dsdt.%
\end{array}
\label{8}
\end{equation}%
\begin{equation}
\begin{array}{l}
\dint\limits_{x}^{b}\dint\limits_{y}^{d}p_{2}(b,t)q_{2}(d,s)\dfrac{\partial
^{2}f(t,s)}{\partial t\partial s}dsdt \\ 
\\ 
=\dint\limits_{x}^{b}p_{2}(b,t)\left[ q_{2}(d,y)\dfrac{\partial f(t,y)}{%
\partial t}-\dint\limits_{y}^{d}w(s)\dfrac{\partial f(t,s)}{\partial t}ds%
\right] dt \\ 
\\ 
=p_{2}(b,x)q_{2}(d,y)f(x,y)+q_{2}(d,y)\dint\limits_{x}^{b}w(t)f(t,y)dt \\ 
\\ 
+p_{2}(b,x)\dint\limits_{y}^{d}w(s)f(x,s)ds+\dint\limits_{x}^{b}\dint%
\limits_{y}^{d}w(s)w(t)f(t,s)dsdt.%
\end{array}
\label{9}
\end{equation}%
Adding (\ref{6})-(\ref{9}) and rewriting, we easily deduce:%
\begin{equation}
\begin{array}{l}
\dint\limits_{a}^{b}\dint\limits_{c}^{d}p(x,t)q(y,s)\dfrac{\partial
^{2}f(t,s)}{\partial t\partial s}dsdt=m(a,b)m(c,d)f(x,y)-m(c,d)\dint%
\limits_{a}^{b}w(t)f(t,y)dt \\ 
\\ 
-m(a,b)\dint\limits_{c}^{d}w(s)f(x,s)ds+\dint\limits_{a}^{b}\dint%
\limits_{c}^{d}w(s)w(t)f(t,s)dsdt%
\end{array}
\label{10}
\end{equation}%
which this completes the proof.
\end{proof}

\begin{theorem}
\label{zz} Let $f:[a,b]\times \lbrack c,d]\mathbb{\rightarrow R}$ be an
absolutely continuous fuction such that the partial derivative of order $2$
exist and is bounded, i.e.,%
\begin{equation*}
\left\Vert \frac{\partial ^{2}f(t,s)}{\partial t\partial s}\right\Vert
_{\infty }=\underset{(x,y)\in (a,b)\times (c,d)}{\overset{}{\sup }}%
\left\vert \frac{\partial ^{2}f(t,s)}{\partial t\partial s}\right\vert
<\infty 
\end{equation*}%
for all $(t,s)\in \lbrack a,b]\times \lbrack c,d].$ Then, we have%
\begin{eqnarray}
&&\left\vert f(x,y)-\left[ \frac{1}{m(a,b)}\dint\limits_{a}^{b}w(t)f(t,y)dt+%
\frac{1}{m(c,d)}\dint\limits_{c}^{d}w(s)f(x,s)ds\right] \right.   \notag \\
&&  \notag \\
&&\ \ \ \ \ \ \ \ \ \ \ \ \ \ \ \ \ \ \ \ \ \ \ \ \ \ \ \ \ \ \ \ \ \left. -%
\frac{1}{m(a,b)m(c,d)}\dint\limits_{a}^{b}\dint%
\limits_{c}^{d}w(s)w(t)f(t,s)dsdt\right\vert   \label{11} \\
&&  \notag \\
&\leq &\frac{A(x)B(y)}{m(a,b)m(c,d)}\left\Vert \frac{\partial ^{2}f(t,s)}{%
\partial t\partial s}\right\Vert _{\infty }  \notag
\end{eqnarray}%
where 
\begin{equation*}
A(x)=\dint\limits_{a}^{x}(x-u)w(u)du+\dint\limits_{x}^{b}(u-x)w(u)du
\end{equation*}%
and%
\begin{equation*}
B(y)=\dint\limits_{c}^{y}(y-u)w(u)du+\dint\limits_{y}^{d}(u-y)w(u)du.
\end{equation*}
\end{theorem}

\begin{proof}
From Lemma \ref{z} and using the properties of modulus, we observe that%
\begin{eqnarray}
&&\left\vert f(x,y)-\left[ \frac{1}{m(a,b)}\dint\limits_{a}^{b}w(t)f(t,y)dt+%
\frac{1}{m(c,d)}\dint\limits_{c}^{d}w(s)f(x,s)ds\right] \right.   \notag \\
&&  \notag \\
&&\left. -\frac{1}{m(a,b)m(c,d)}\dint\limits_{a}^{b}\dint%
\limits_{c}^{d}w(s)w(t)f(t,s)dsdt\right\vert   \notag \\
&&  \notag \\
&\leq &\frac{1}{m(a,b)m(c,d)}\dint\limits_{a}^{b}\dint\limits_{c}^{d}\left%
\vert p(x,t)\right\vert \left\vert q(y,s)\right\vert \left\vert \frac{%
\partial ^{2}f(t,s)}{\partial t\partial s}\right\vert dsdt  \label{01} \\
&&  \notag \\
&\leq &\frac{1}{m(a,b)m(c,d)}\left\Vert \frac{\partial ^{2}f(t,s)}{\partial
t\partial s}\right\Vert _{\infty
}\dint\limits_{a}^{b}\dint\limits_{c}^{d}\left\vert p(x,t)\right\vert
\left\vert q(y,s)\right\vert dsdt.  \notag
\end{eqnarray}%
Now, using the change of order of integration we get%
\begin{eqnarray}
\dint\limits_{a}^{b}\left\vert p(x,t)\right\vert dt
&=&\dint\limits_{a}^{x}\left\vert p_{1}(a,t)\right\vert
dt+\dint\limits_{x}^{b}\left\vert p_{2}(b,t)\right\vert dt  \notag \\
&&  \notag \\
&\leq
&\dint\limits_{a}^{x}\dint\limits_{a}^{t}w(u)dudt+\dint\limits_{x}^{b}\dint%
\limits_{t}^{b}w(u)dudt  \notag \\
&&  \label{12} \\
&=&\dint\limits_{a}^{x}w(u)\dint\limits_{u}^{x}dtdu+\dint\limits_{x}^{b}w(u)%
\dint\limits_{x}^{u}dtdu  \notag \\
&&  \notag \\
&=&\dint\limits_{a}^{x}(x-u)w(u)du+\dint\limits_{x}^{b}(u-x)w(u)du  \notag
\end{eqnarray}%
and similarly,%
\begin{eqnarray}
\dint\limits_{c}^{d}\left\vert q(y,s)\right\vert ds
&=&\dint\limits_{c}^{y}\left\vert q_{1}(c,s)\right\vert
ds+\dint\limits_{y}^{d}\left\vert q_{2}(d,s)\right\vert ds  \notag \\
&&  \label{13} \\
&\leq &\dint\limits_{c}^{y}(y-u)w(u)du+\dint\limits_{y}^{d}(u-y)w(u)du. 
\notag
\end{eqnarray}%
Thus, using (\ref{12}) and (\ref{13}) in (\ref{01}), we obtain the
inequality (\ref{11}) and the proof is completed.
\end{proof}

\begin{remark}
If we choose $w(u)=1$ in Theorem \ref{zz}$,$ then the inequality (\ref{11})
reduces the inequality (\ref{2}) which is proved by Barnett and Dragomir in 
\cite{Barnett}.
\end{remark}

\begin{corollary}
\label{z2} Under the assumptions of Theorem \ref{zz}, we have%
\begin{multline*}
\left\vert f(\frac{a+b}{2},\frac{c+d}{2})-\left[ \frac{1}{m(a,b)}%
\dint\limits_{a}^{b}w(t)f(t,\frac{c+d}{2})dt+\frac{1}{m(c,d)}%
\dint\limits_{c}^{d}w(s)f(\frac{a+b}{2},s)ds\right] \right.  \\
\\
\left. -\frac{1}{m(a,b)m(c,d)}\dint\limits_{a}^{b}\dint%
\limits_{c}^{d}w(s)w(t)f(t,s)dsdt\right\vert  \\
\\
\leq \frac{A(\frac{a+b}{2})B(\frac{c+d}{2})}{m(a,b)m(c,d)}\left\Vert \frac{%
\partial ^{2}f(t,s)}{\partial t\partial s}\right\Vert _{\infty }
\end{multline*}%
where 
\begin{equation*}
A(\frac{a+b}{2})=\dint\limits_{a}^{\frac{a+b}{2}}(\frac{a+b}{2}%
-u)w(u)du+\dint\limits_{\frac{a+b}{2}}^{b}(u-\frac{a+b}{2})w(u)du
\end{equation*}%
and%
\begin{equation*}
B(\frac{c+d}{2})=\dint\limits_{c}^{\frac{c+d}{2}}(\frac{c+d}{2}%
-u)w(u)du+\dint\limits_{\frac{c+d}{2}}^{d}(u-\frac{c+d}{2})w(u)du.
\end{equation*}
\end{corollary}

\begin{remark}
We choose $w(u)=1$ in Corollary \ref{z2}, we get%
\begin{multline*}
\left\vert f(\frac{a+b}{2},\frac{c+d}{2})-\left[ \frac{1}{b-a}%
\dint\limits_{a}^{b}f(t,\frac{c+d}{2})dt+\frac{1}{d-c}\dint\limits_{c}^{d}f(%
\frac{a+b}{2},s)ds\right] \right. \\
\left. -\frac{1}{(b-a)(d-c)}\dint\limits_{a}^{b}\dint%
\limits_{c}^{d}f(t,s)dsdt\right\vert \\
\leq \frac{(b-a)(d-c)}{16}\left\Vert \frac{\partial ^{2}f(t,s)}{\partial
t\partial s}\right\Vert _{\infty }.
\end{multline*}
\end{remark}

\begin{remark}
We choose $w(u)=u$ in Corollary \ref{z2}, we get%
\begin{multline*}
\left\vert f(\frac{a+b}{2},\frac{c+d}{2})-\left[ \frac{2}{(b-a)^{2}}%
\dint\limits_{a}^{b}tf(t,\frac{c+d}{2})dt+\frac{2}{(d-c)^{2}}%
\dint\limits_{c}^{d}sf(\frac{a+b}{2},s)ds\right] \right. \\
\\
\left. -\frac{4}{(b-a)^{2}(d-c)^{2}}\dint\limits_{a}^{b}\dint%
\limits_{c}^{d}tsf(t,s)dsdt\right\vert \\
\\
\leq \frac{(a+b)(c+d)}{16}\left\Vert \frac{\partial ^{2}f(t,s)}{\partial
t\partial s}\right\Vert _{\infty }.
\end{multline*}
\end{remark}

\begin{lemma}
\label{s1} Let $f:[a,b]\times \lbrack c,d]\mathbb{\rightarrow R}$ be an
absolutely continuous fuction such that the partial derivative of order $2$
exist for all $(t,s)\in \lbrack a,b]\times \lbrack c,d].$ Then, we have%
\begin{eqnarray}
f(x,y) &=&\frac{1}{m(\alpha _{1},\alpha _{2})}\dint\limits_{\alpha
_{1}}^{\alpha _{2}}w(t)f(t,y)dt+\frac{1}{m(\beta _{1},\beta _{2})}%
\dint\limits_{\beta _{1}}^{\beta _{2}}w(s)f(x,s)ds  \notag \\
&&  \label{14} \\
&&-\frac{1}{m(\alpha _{1},\alpha _{2})m(\beta _{1},\beta _{2})}\left[
\dint\limits_{\alpha _{1}}^{\alpha _{2}}\dint\limits_{\beta _{1}}^{\beta
_{2}}w(t)w(s)f(t,s)dsdt-\dint\limits_{\alpha _{1}}^{\alpha
_{2}}\dint\limits_{\beta _{1}}^{\beta _{2}}P(x,t)Q(y,s)\frac{\partial
^{2}f(t,s)}{\partial t\partial s}dsdt\right]  \notag
\end{eqnarray}%
where%
\begin{equation*}
P(x,t)=\left\{ 
\begin{array}{ll}
P_{1}(\alpha _{1},t)=\dint\limits_{\alpha _{1}}^{t}w(u)du, & \alpha _{1}\leq
t<x \\ 
&  \\ 
P_{2}(\alpha _{2},t)=\dint\limits_{\alpha _{2}}^{t}w(u)du, & x\leq t\leq
\alpha _{2}%
\end{array}%
\right.
\end{equation*}%
and%
\begin{equation*}
Q(y,s)=\left\{ 
\begin{array}{ll}
Q_{1}(\beta _{1},s)=\dint\limits_{\beta _{1}}^{s}w(u)du, & \beta _{1}\leq s<y
\\ 
&  \\ 
Q_{2}(\beta _{2},s)=\dint\limits_{\beta _{2}}^{s}w(u)du, & y\leq s\leq \beta
_{2}.%
\end{array}%
\right.
\end{equation*}%
for all $(\alpha _{i},\beta _{i})\in \lbrack a,b]\times \lbrack c,d],\ i=1,2$
with $\alpha _{1}<\alpha _{2},\ \beta _{1}<\beta _{2}.$
\end{lemma}

\begin{proof}
By definitions of $P(x,t)$ and $Q(y,s),$ we have%
\begin{equation*}
\begin{array}{l}
\dint\limits_{\alpha _{1}}^{\alpha _{2}}\dint\limits_{\beta _{1}}^{\beta
_{2}}P(x,t)Q(y,s)\dfrac{\partial ^{2}f(t,s)}{\partial t\partial s}%
dsdt=\dint\limits_{\alpha _{1}}^{x}\dint\limits_{\beta _{1}}^{y}P_{1}(\alpha
_{1},t)Q_{1}(\beta _{1},s)\dfrac{\partial ^{2}f(t,s)}{\partial t\partial s}%
dsdt \\ 
\\ 
+\dint\limits_{\alpha _{1}}^{x}\dint\limits_{y}^{\beta _{2}}P_{1}(\alpha
_{1},t)Q_{2}(\beta _{2},s)\dfrac{\partial ^{2}f(t,s)}{\partial t\partial s}%
dsdt+\dint\limits_{x}^{\alpha _{2}}\dint\limits_{\beta _{1}}^{y}P_{2}(\alpha
_{2},t)Q_{1}(\beta _{1},s)\dfrac{\partial ^{2}f(t,s)}{\partial t\partial s}%
dsdt \\ 
\\ 
+\dint\limits_{x}^{\alpha _{2}}\dint\limits_{y}^{\beta _{2}}P_{2}(\alpha
_{2},t)Q_{2}(\beta _{2},s)\dfrac{\partial ^{2}f(t,s)}{\partial t\partial s}%
dsdt.%
\end{array}%
\end{equation*}%
By similar computation in Lemma \ref{z}, we get (\ref{14}).
\end{proof}

\begin{theorem}
\label{s2} Let $f:[a,b]\times \lbrack c,d]\mathbb{\rightarrow R}$ be an
absolutely continuous fuction such that the partial derivative of order $2$
exist and is bounded, i.e.,%
\begin{equation*}
\left\Vert \frac{\partial ^{2}f(t,s)}{\partial t\partial s}\right\Vert
_{\infty }=\underset{(x,y)\in (a,b)\times (c,d)}{\overset{}{\sup }}%
\left\vert \frac{\partial ^{2}f(t,s)}{\partial t\partial s}\right\vert
<\infty
\end{equation*}%
for all $(t,s)\in \lbrack a,b]\times \lbrack c,d].$ Then, we have%
\begin{eqnarray}
&&\left\vert f(x,y)-\left[ \frac{1}{m(\alpha _{1},\alpha _{2})}%
\dint\limits_{\alpha _{1}}^{\alpha _{2}}w(t)f(t,y)dt+\frac{1}{m(\beta
_{1},\beta _{2})}\dint\limits_{\beta _{1}}^{\beta _{2}}w(s)f(x,s)ds\right]
\right.  \notag \\
&&  \notag \\
&&\ \ \ \ \ \ \ \ \ \ \ \ \ \ \ \ \ \ \ \ \ \ \ \ \left. -\frac{1}{m(\alpha
_{1},\alpha _{2})m(\beta _{1},\beta _{2})}\dint\limits_{\alpha _{1}}^{\alpha
_{2}}\dint\limits_{\beta _{1}}^{\beta _{2}}w(t)w(s)f(t,s)dsdt\right\vert
\label{15} \\
&&  \notag \\
&\leq &\frac{A_{1}(x)B_{1}(y)}{m(\alpha _{1},\alpha _{2})m(\beta _{1},\beta
_{2})}\left\Vert \frac{\partial ^{2}f(t,s)}{\partial t\partial s}\right\Vert
_{\infty }  \notag
\end{eqnarray}%
where 
\begin{equation*}
A_{1}(x)=\dint\limits_{\alpha _{1}}^{x}(x-u)w(u)du+\dint\limits_{x}^{\alpha
_{2}}(u-x)w(u)du
\end{equation*}%
and%
\begin{equation*}
B_{1}(y)=\dint\limits_{\beta _{1}}^{y}(y-u)w(u)du+\dint\limits_{y}^{\beta
_{2}}(u-y)w(u)du
\end{equation*}%
for all $(\alpha _{i},\beta _{i})\in \lbrack a,b]\times \lbrack c,d],\ i=1,2$
with $\alpha _{1}<\alpha _{2},\ \beta _{1}<\beta _{2}.$
\end{theorem}

\begin{proof}
From Lemma \ref{z}, using the properties of modulus and by similar
computation in Theorem \ref{zz}, we get (\ref{15}). The details are omitted.
\end{proof}

\begin{remark}
We choose $\alpha _{1}=a,\ \alpha _{2}=b,\ \beta _{1}=c,\ \beta _{2}=d\ $ in
(\ref{15}), then Theorem \ref{s2} reduces Theorem \ref{zz}.
\end{remark}

\begin{corollary}
\label{s3} Under the assumptions of Theorem \ref{s2}, we have%
\begin{multline*}
\left\vert f(\frac{a+b}{2},\frac{c+d}{2})-\left[ \frac{1}{m(\alpha
_{1},\alpha _{2})}\dint\limits_{\alpha _{1}}^{\alpha _{2}}w(t)f(t,\frac{c+d}{%
2})dt+\frac{1}{m(\beta _{1},\beta _{2})}\dint\limits_{\beta _{1}}^{\beta
_{2}}w(s)f(\frac{a+b}{2},s)ds\right] \right. \\
\\
\left. -\frac{1}{m(\alpha _{1},\alpha _{2})m(\beta _{1},\beta _{2})}%
\dint\limits_{\alpha _{1}}^{\alpha _{2}}\dint\limits_{\beta _{1}}^{\beta
_{2}}w(t)w(s)f(t,s)dsdt\right\vert \\
\\
\leq \frac{A_{1}(\frac{a+b}{2})B_{1}(\frac{c+d}{2})}{m(\alpha _{1},\alpha
_{2})m(\beta _{1},\beta _{2})}\left\Vert \frac{\partial ^{2}f(t,s)}{\partial
t\partial s}\right\Vert _{\infty }
\end{multline*}%
where 
\begin{equation*}
A_{1}(\frac{a+b}{2})=\dint\limits_{\alpha _{1}}^{\frac{a+b}{2}}(\frac{a+b}{2}%
-u)w(u)du+\dint\limits_{\frac{a+b}{2}}^{\alpha _{2}}(u-\frac{a+b}{2})w(u)du
\end{equation*}%
and%
\begin{equation*}
B_{1}(\frac{c+d}{2})=\dint\limits_{\beta _{1}}^{\frac{c+d}{2}}(\frac{c+d}{2}%
-u)w(u)du+\dint\limits_{\frac{c+d}{2}}^{\beta _{2}}(u-\frac{c+d}{2})w(u)du.
\end{equation*}
\end{corollary}

\begin{remark}
We choose $w(u)=1$ in Corollary \ref{s3}, we get%
\begin{multline*}
\left\vert f(\frac{a+b}{2},\frac{c+d}{2})-\left[ \frac{1}{(\alpha
_{2}-\alpha _{1})}\dint\limits_{\alpha _{1}}^{\alpha _{2}}f(t,\frac{c+d}{2}%
)dt+\frac{1}{(\beta _{2}-\beta _{1})}\dint\limits_{\beta _{1}}^{\beta _{2}}f(%
\frac{a+b}{2},s)ds\right] \right. \\
\\
\left. -\frac{1}{(\alpha _{2}-\alpha _{1})(\beta _{2}-\beta _{1})}%
\dint\limits_{\alpha _{1}}^{\alpha _{2}}\dint\limits_{\beta _{1}}^{\beta
_{2}}f(t,s)dsdt\right\vert \\
\\
\leq \frac{A_{3}B_{3}}{64(\alpha _{2}-\alpha _{1})(\beta _{2}-\beta _{1})}%
\left\Vert \frac{\partial ^{2}f(t,s)}{\partial t\partial s}\right\Vert
_{\infty }
\end{multline*}%
where%
\begin{equation*}
A_{3}=(a+b-2\alpha _{1})^{2}+(a+b-2\alpha _{2})^{2}
\end{equation*}%
and%
\begin{equation*}
B_{3}=(c+d-2\beta _{1})^{2}+(c+d-2\beta _{2})^{2}.
\end{equation*}
\end{remark}

\end{document}